\documentclass[12pt]{amsart}

\usepackage{amsmath,amssymb,amsfonts}

\newtheorem{theorem}{Theorem}[section]
\newtheorem{lemma}[theorem]{Lemma}

\newtheorem{corollary}[theorem]{Corollary}

\newtheorem*{example}{Example}
\newtheorem{thm}{Theorem}
\newtheorem{corA}{Corollary}

\newtheorem{corB}{Corollary}

\begin{document}

\title[Contranormal-free groups]{On the structure of some contranormal-free groups}

\author[M. Dixon]{Martyn R. Dixon}
\address[Martyn R. Dixon]
{Department of Mathematics\\
University of Alabama\\
Tuscaloosa, AL 35487-0350, U.S.A.}
\email{mdixon@ua.edu}
\author[L.  Kurdachenko]{Leonid A. Kurdachenko}
\address[Leonid A. Kurdachenko]
{Department of Algebra, Facultet of mathematic and mechanik\\
National University of Dnepropetrovsk\\
Gagarin prospect 72\\
Dnipro 10, 49010, Ukraine.}
\email{lkurdachenko@i.ua, lkurdachenko@gmail.com}
\author[I. Subbotin]{Igor Ya. Subbotin}
\address[Igor Ya. Subbotin]{Department of Mathematics and Natural Sciences, Sanford National University\\
5245 Pacific Concourse Drive,  \\
Los Angeles, CA 90045-6904, USA.}
\email{isubboti@nu.edu}

\begin{abstract}
A subgroup $H$ of a group $G$ is contranormal if $H^G=G$. In finite groups if there are no proper contranormal subgroups, then the group is nilpotent but this is not true in infinite groups as the well-known Heineken-Mohamed groups show.  We call such groups without proper contranormal subgroups ``contranormal-free''.  In this paper we prove various results concerning contranormal-free groups proving, for example that locally generalized radical contranormal-free groups which have finite section rank  are hypercentral.
\end{abstract} 

%\thanks{The authors would like to the thank the referees for some helpful suggestions.}

\keywords{contanormal subgroup, hypercentral group, nilpotent group, finitely generated group}

\subjclass[2010]{Primary: 20E15; Secondary 20F16, 20F22}

\maketitle

\newcommand{\sub}[2]{\langle#1,#2\rangle}%subgp gen by two elements
\newcommand{\cyclic}[1]{\langle #1\rangle}%cyclic group
\newcommand{\norm}{\triangleleft\,}%normal subgroup
\newcommand{\core}[2]{\text{core}\,_{#1}\,#2}%the core of a subgroup
\newcommand{\cpin}{C_{p^{\infty}}}%Cpinfinity
\newcommand{\co}[1]{\,\scriptstyle{\mathbf{#1}}\textstyle}% closure op 
\newcommand{\cogo}[2]{\co{#1}\mathfrak{#2}}%closure op and class
\newcommand{\scco}[1]{\,\scriptscriptstyle{\mathbf{#1}}\scriptstyle}
\newcommand{\sccogo}[2]{\scco{#1}\mathfrak{#2}}
\newcommand{\sccogoln}{\scco{L}\mf{N}}
\newcommand{\ann}{\textbf{Ann}}%annihilator
\newcommand{\spec}{\text{Spec\,}}%spectrum
\newcommand{\listel}[2]{#1_1,\dots, #1_{#2}}%list of elements
\newcommand{\aut}{\text{Aut}\,}%aut G.
\newcommand{\cent}[2]{C_{#1}(#2)}%centralizer
\newcommand{\Dr}{\text{Dr}\,}%direct product
\newcommand{\dir}[3]{\underset{#1\in #2}{\Dr}#3_{#1}}%direct product
\newcommand{\nocl}[2]{\langle #1\rangle^{#2}}%normal closure of element
\newcommand{\qn}{\,\text{qn}\,}%quasinormal
\newcommand{\comm}[2]{#1^{-1}#2^{-1}#1#2}%commutator
\newcommand{\conj}[2]{#1^{-1}#2#1}%conjugate
\newcommand{\izer}[2]{N_{#1}(#2)}%normalizer
\newcommand{\gl}{\text{GL\,}}%general linear
\newcommand{\subgp}[3]{\langle#1_{#2}\mid #2\in#3\rangle}%subgp
\newcommand{\direct}[3]{\underset{#1\in #2}{\Dr}\langle#3_{#1}\rangle}
\newcommand{\omn}[1]{\Omega_n{(#1)}}
\newcommand{\Hom}{\text{Hom}\,}
\newcommand{\Cr}[3]{\underset{#1\in #2}{\text{Cr}}\,#3_{#1}}%Cart prod
\newcommand{\inflist}[1]{\{#1_1,#1_2,\dots\}} %infinite countable list
\newcommand{\cart}{\text{Cr}\,} %Cartesian product
\newcommand{\carti}[3]{\underset{#1\geq #2}\cart #3_{#1}} %cart product
\newcommand{\fgsub}[2]{\langle #1_1,\dots,#1_{#2}\rangle}%finitely gen subgp
\newcommand{\cwr}{\,\bar{\wr}\,}%complete wreath product
\newcommand{\minn}{\infty$-$\overline{\mathfrak{N}}}
\newcommand{\mnn}{\overline{\mathfrak{N}}}
\newcommand{\mnp}{\overline{\mathcal{P}}}
\newcommand{\mns}{\overline{\mathfrak{S}}}
\newcommand{\minsd}{\infty$-$\overline{\mathfrak{S}_d}}
\newcommand{\minnc}{\infty$-$\overline{\mathfrak{N}_c}}
\newcommand{\ms}{\mathfrak{S}}
\newcommand{\mn}{\mathfrak{N}}
\newcommand{\mf}[1]{\mathfrak{#1}}
\newcommand{\sdr}{\mathfrak{S}_d(r)}
\newcommand{\sd}{\mathfrak{S}_d}
\newcommand{\mfr}{\mathfrak{R}}
\newcommand{\ncr}{\mathfrak{N}_c(r)}
\newcommand{\mfsr}{\overline{\ms\mfr}}
\newcommand{\lslf}{\overline{(\cogo{L}{S})(\cogo{L}{F})}}
\newcommand{\mcp}{\mathcal{P}}
\newcommand{\mcps}{\mathcal{P}^*}
\newcommand{\mfas}{\mathfrak{A}^*}
\newcommand{\mfa}{\mathfrak{A}}
\newcommand{\mfss}{\mathfrak{S}^*}
\newcommand{\mfns}{\mathfrak{N}^*}
\newcommand{\wmcnq}{min-$\infty$-$\overline{\text{qn}}$ }
\newcommand{\mfx}{\mathfrak{X}}
\newcommand{\frat}{\text{Frat}}
\newcommand{\mrn}{\mathfrak{R}\mathfrak{N}}
\newcommand{\dmn}{dim}
\newcommand{\rzero}{\textbf{r}_0}
\newcommand{\dl}{\textbf{dl}}
\newcommand{\srp}{\textbf{sr}_p}
\newcommand{\mfy}{\mathfrak{Y}}
\newcommand{\mln}{\text{$\mathbf{L}$}\mathfrak{N}}
\newcommand{\tor}{\textbf{Tor}}
\newcommand{\omegn}[2]{\boldsymbol{\Omega}_{#1}(#2)} 
\newcommand{\cogoln}{\co{L}\mf{N}}
\newcommand{\ldl}{\leq\dots\leq}
\newcommand{\sprk}{\textbf{r}}
\newcommand{\cl}{\textbf{cl}}
\newcommand{\bfd}{\textbf{d}}
\newcommand{\zl}{\textbf{zl}}
\newcommand{\mfn}{\mathfrak{N}}
\newcommand{\rz}{\mathbf{r}_{\mathbb{Z}}}%torsion-free rank
\newcommand{\dimbf}{\textbf{dim}}
\newcommand{\kerbf}{\textbf{ker}}
\newcommand{\imbf}{\textbf{Im}}
\newcommand{\short}[3]{#1=#1_0\leq #1_1\leq\dots\leq #1_{#2}=#3} %short non-normal series
\newcommand{\mmx}{\textbf{mmx}}
\newcommand{\bfs}{\textbf{s}}
\newcommand{\bfo}{\textbf{O}}
\newcommand{\zr}{\textbf{r}_{\text{Z}}}
\newcommand{\bs}{\textbf{bs}}
\newcommand{\shortserone}[3]{1=#1_0\norm #1_1\norm\dots\norm #1_{#2}=#3} %eg nilpotent
\newcommand{\agemo}[2]{\boldsymbol{\mho}_{#1}(#2)} 
\newcommand{\bfl}{\textbf{l}}
\newcommand{\zrp}{\textbf{r}_{\text{Z},p}} %zaitsev rank of p-component
\newcommand{\bflo}{\textbf{o}}
\newcommand{\bfr}{\textbf{r}}
\newcommand{\bffc}{\textbf{FC}}
\newcommand{\bfe}{\textbf{E}}
\newcommand{\neng}[3]{\bfe_{#1,#2}(#3)}
\newcommand{\mcal}[1]{\mathcal{#1}}
\newcommand{\centdim}{\textbf{centdim}}
\newcommand{\cobf}{\textbf{codim}}
\newcommand{\augdim}{\textbf{augdim}}
\newcommand{\codim}{\textbf{codim}}
\newcommand{\chl}{\textbf{chl}}
\newcommand{\cha}{\textbf{char}} 
%%%%%%%%%%%%%%%%%%%%%%greek
\newcommand{\al}{\alpha}
\newcommand{\de}{\delta}
\newcommand{\la}{\lambda}
\newcommand{\ga}{\gamma}
\newcommand{\Ga}{\Gamma}
\newcommand{\Si}{\Sigma}
\newcommand{\La}{\Lambda}
\newcommand{\ka}{\kappa}
\newcommand{\De}{\Delta}
\newcommand{\Th}{\Theta}
\newcommand{\ze}{\zeta}
\newcommand{\om}{\omega}
\newcommand{\be}{\beta}
\newcommand{\si}{\sigma}
\newcommand{\Om}{\Omega}
\newcommand{\ass}{\textbf{Ass}}
%%%%%%%%%%%%%%%%%%%%%%%%%%%%%%%%%
\newcommand{\fn}[3]{#1: #2\longrightarrow #3} % function from one set to another
\newcommand{\bigset}[3]{\{#1_{#2} | #2\in #3\}}
\newcommand{\mat}[2]{\text{Mat}_{#1}(#2)}
\newcommand{\mx}[4]{[#1_{#2 #3}]_{#2, #3\in #4}}
\newcommand{\mxel}[3]{#1_{#2 #3}}
\newcommand{\censeries}[4]{0=\ze_{{#1},0}(#2)\leq \ze_{{#1},1}(#2)\leq\dots \ze_{{#1},{#3}}(#2)\leq \ze_{{#1}, {#3} +1}(#2)\leq\dots \ze_{{#1}, {#4}}(#2)} % central series
\newcommand{\modseries}[3]{0=#1_0\leq #1_1\leq\dots #1_{#2}\leq #1_{#2 +1}\leq\dots #1_{#3}} %series
\newcommand{\shortchain}[3]{0=#1_0\leq #1_1\leq\dots\leq #1_{#2}=#3} %subnormal series
\newcommand{\descseries}[3]{#1_0 \geq #1_1\geq\dots\geq #1_{#2}\geq #1_{#2+1}\geq\dots \bigcap_{#2\in #3}#1_{#2}}
\newcommand{\lowcenser}[4] {\ga_{#1, 1}(#2)\geq \ga_{#1,2}(#2)\geq \dots \ga_{#1, #3}(#2)\geq \ga_{#1,#3+1}(#2)\geq \dots \ga_{#1,#4}(#2)}
\newcommand{\normbf}[2]{\textbf{Norm}_{#1}(#2)}
\newcommand{\charbf}{\textbf{char}\,}
\newcommand{\inv}{\textbf{Inv}}
\newcommand{\mon}{\textbf{mon}}

\section{Introduction} \label{s:intro}

In this  paper we consider certain subgroups of a group $G$ whose properties have not yet been studied in great detail, especially in infinite groups, but which, nevertheless, have significant control over the group structure.

If $H$ is a subgroup of a group $G$, then as usual we define the normal closure of $H$ in $G$ to be the smallest normal subgroup of $G$ which contains $H$. We denote the normal closure of $H$ in $G$ by $H^G$ and clearly $H^G=\cyclic{H^x| x\in G}$.  Naturally the subgroup $H$ is normal in $G$ if $H^G=H$ and in this sense the subgroups $H$ for which $H^G=G$ are complete opposites of normal subgroups.

Following J. S. Rose~\cite{jR68} we shall call a subgroup  $H$  of a group  $G$    \emph{contranormal}  in  $G$ if $H^G=G$.

In some sense contranormal subgroups spread their influence over the entire group: all elements of a group are in their normal closures  while all other subgroups have smaller spheres of influence.  In this regard,  contranormal subgroups are opposite to normal and subnormal subgroups. It is clear that a contranormal subgroup $H$ of a group $G$ is normal (or subnormal) if and only if $H=G$.

Contranormal subgroups occur naturally in the following way. If $G$ is a group and $H$ is a
subgroup of $G$ we can construct the following canonical series. 

We let $v_{0,G}(H)=G, v_{1,G}(H)=H^G$ and define $v_{\al+1,G}(H)=H^{v_{\al, G}(H)}$ for every ordinal $\al$ and $v_{\la,G}(H)=\cap_{\be<\la} v_{\be,G}(H)$ for all limit ordinals $\la$.

We may then construct the \emph{lower normal closure series}
\[
G=v_{0,G}(H)\geq v_{1,G}(H)\geq \dots v_{\al, G}(H)\geq v_{\al+1,G}(H)\geq \dots v_{\ga,G}(H)=D
\]
of a subgroup  $H$  in the group  $G$. By  construction  $ v_{\al+1,G}(H)$ is a normal subgroup of     
$v_{\al,G}(H)$ for all ordinals  $\al <\ga$.  The last term  $D$  of this series has the property that   $H^D=D$. The last term  of this series, $v_{\ga,G}(H)$, is called the \emph{lower normal  closure  of  $H$  in  $G$}.

We note that if $H$ is a subgroup of $G$, then $H$ is contranormal in its lower normal closure in $G$. Furthermore the lower normal closure of $H$ in $G$ is a descendant subgroup of $G$.

Important examples of  contranormal subgroups include abnormal subgroups. We recall that a subgroup $H$ of a group $G$ is \emph{abnormal} in $G$  if $g\in \cyclic{H,H^g}$ for all elements $g\in G$.

However, these two types of subgroup do not coincide and the influence of these subgroups on the structure of groups is different.

When $G$ is a finite nilpotent group,  it is clear that $G$ contains no proper contranormal subgroups; furthermore, if $G$ has no proper contranormal subgroups, then the maximal subgroups of $G$ must be normal and in this case $G$ is nilpotent (see \cite[Satz III.3.11]{bH67}).  However, for infinite groups the situation is different. If $G$ is an infinite locally nilpotent group, then $G$ does not contain proper abnormal subgroups, whereas a locally nilpotent group can contain proper contranormal subgroups; indeed there are even hypercentral groups with proper contranormal subgroups. 

For example, let $D$ be a divisible abelian $2$-group and let $\al$ denote the inversion automorphism of $D$. Let $G=D\rtimes \cyclic{\al}$. It is clear that $G$ is a hypercentral, abelian-by-finite group and it is easy to see that $\cyclic{\al}$ is contranormal in $G$. As another example we can let $D$ be a Pr\"{u}fer $2$-group and let $\al$ be an automorphism of $D$ of infinite order. Then $G=D\rtimes \cyclic{\al}$ is hypercentral and again $\cyclic{\al}$ is contranormal in $G$.  Since every subgroup of a hypercentral group is ascendant it is clear that in either case $\cyclic{\al}$ is an ascendant subgroup of $G$.

The following natural question arises.

\begin{quote}
What can be said concerning a group $G$ which contains no proper contranormal subgroups?
\end{quote}

We shall say that a group $G$ is \emph{contranormal-free} if $G$ contains no proper contranormal subgroups. (In \cite{bW20} the term ``Rose-nilpotent'' is used.)

Various types of contranormal-free groups have been studied in the papers \cite{KOS09b,KOS09,  KS03b,  bW20}.  That such groups need not even be locally nilpotent is exhibited by examples in \cite{KOS09b} and \cite{bW20} where a general method is given for constructing such groups.  For the most part though these papers have been concerned with obtaining results showing that certain types of contranormal-free groups are nilpotent.  In this paper we continue the study of contranormal-free groups but now direct attention to obtaining conditions for a contranormal-free group to be  hypercentral.  Our main results are given in the following theorems. First we obtain:

\begin{thm}\label{t:1}
Let $G$ be a group and let $L$ be a nilpotent normal subgroup of $G$ such that $G/L$ is a finitely generated hyper (abelian or finite) group.  If $G$ is contranormal-free, then $G$ is hypercentral.
\end{thm}

%We give an example later in the article to show that in the situation of Theorem~\ref{t:1}, the group %$G$ need not be nilpotent. 

We next recall that a group $G$ is \emph{generalized radical} if it has an ascending series whose factors are locally nilpotent or locally finite.

Let $p$ be a prime. We say that $G$ has \emph{finite section $p$-rank} $\srp(G)=r$ if every elementary abelian $p$-section of $G$ is finite of order at most $p^r$ and there is an elementary abelian $p$-section $A/B$ of $G$ such that $|A/B|=p^r$.

We say that the group $G$ has \emph{finite section rank} if $\srp(G)$ is finite for each prime $p$.

We let $\om$ denote the first infinite ordinal. If $G$ is a group, then we let $\Pi(G)$ denote the set of primes $p$ for which $G$ has an element of order $p$.

We may now state our next main result.

\begin{thm}\label{t:2}
Let $G$ be a locally generalized radical group with finite section rank. If $G$ is contranormal-free, then $G$ is a hypercentral group of hypercentral length at most $\om +k$ for some natural number $k$. Moreover, if $H$  is normal in $G$ and  $\Pi(G/H)$ is finite,  then $G/H$ is nilpotent.
\end{thm}

A group $G$ has finite special rank $\bfr(G)=r$ if every finitely generated subgroup of $G$ can be generated by $r$ elements and $r$ is the least natural number with this property.

The following corollary is immediate.

\begin{corB}\label{c:3}
Let $G$ be a locally generalized radical group with finite special rank. If $G$ is contranormal-free, then $G$ is a hypercentral group of hypercentral length at most $\om +k$ for some natural number $k$. Moreover, if $H$  is normal in $G$ and  $\Pi(G/H)$ is finite,  then $G/H$ is nilpotent.
\end{corB}

Clearly a contranormal-free group of finite special rank need not be nilpotent.

%%%%%%%%%%%%%%%%%%%%%%%%%%%%%

\section[Preliminary Results]{Preliminary Results} \label{s:fininf}

We begin with the following elementary result which is often useful.

\begin{lemma}\label{l:1}
Let $G$ be a group. 
\begin{enumerate}
\item[  (i)] If $C$ is a contranormal subgroup of $G$ and $K$ is a subgroup containing $C$, then $K$ is a contranormal subgroup of $G$.
\item[ (ii)] If $C$ is a contranormal subgroup of $G$ and $H$ is a normal subgroup of $G$, then $CH/H$  is a contranormal subgroup of $G/H$.
\item[(iii)] If $H$ is a normal subgroup of $G$, $C$ is a subgroup of $G$ such that $H\leq C$ and $C/H$  is a contranormal subgroup of  $G/H$,  then $C$ is a contranormal subgroup  of $G$.
\item[ (iv)] If $C$ is a contranormal subgroup of $G$ and if $D$ is a contranormal subgroup of $C$, then $D$ is a contranormal subgroup of $G$.
\item[  (v)] If $M$ is a maximal subgroup of $G$ and $M$ is not normal in $G$,  then $M$ is a contranormal subgroup of $G$.
\end{enumerate}
\end{lemma}

We shall use the following fact regularly and note that it is immediate from \cite[5.1.6]{dR96}.

\begin{lemma}\label{l:2}
Let $G$ be a group, let $A$ be normal subgroup of $G$ and let $S$ be a subgroup of $G$. Suppose that  $C=SA$.  Then $S^C=S[S,A]$.
\end{lemma}

\begin{lemma}\label{l:3}
Let $G$ be a group and let $A$ be a normal subgroup of $G$. Suppose that $A$ has an ascending series of $G$-invariant subgroups
\[
1=A_0\leq A_1\leq \dots\leq  A_{\al}\leq A_{\al+1}\leq \dots A_{\ga}=A
\]
such that $[G, A_{\al+1}/A_{\al}]=A_{\al+1}/A_{\al}$ for all $\al<\ga$. Then $[G,A]=A$.
\end{lemma}

\begin{proof}
The stated condition implies that $[G,A_{\al+1}]A_{\al}=A_{\al+1}$ for all $\al<\ga$.
We use transfinite induction on $\al$ to prove that $[G,A_{\al}]=A_{\al}$ for all $\al<\ga$. The case $\al=1$ is clear, so now suppose that $[G, A_{\be}]=G_{\be}$ for all $\be<\al$.

If $\al$ is a limit ordinal, then $A_{\al}=\cup_{\be<\al}A_{\be}$, so if $d\in A_{\al}$ there exists $\be<\al$ such that $d\in A_{\be} $ and hence $d\in [G, A_{\be}]$. Thus $d\in [G,A_{\al}]$ and then $A_{\al}=[G,A_{\al}]$ as required.

If $\al-1$ exists, then we have $[G, A_{\al}]A_{\al-1}=A_{\al}$ and $[G,A_{\al-1}]=A_{\al-1}$. Consequently we have
\[
[G,A_{\al}]=[G,A_{\al}A_{\al-1}]=[G,A_{\al}][G,A_{\al-1}]=[G,A_{\al}]A_{\al-1}=A_{\al}.
\]

In particular, when $\al=\ga$ we have $[G,A]=A$ as required.
\end{proof}

If $G$ is a group and $B,C$ are normal subgroups of $G$ such that $B\leq C$, then the factor $C/B$ is called \emph{$G$-central} if $C_G(C/B)=G$, or equivalently $[G,C]\leq B$. Also $C/B$ is called \emph{$G$-eccentric} if $C_G(C/B)\neq G$.

Let $A$ be a normal subgroup of $G$. The \emph{upper $G$-central series} of $A$ is the series
\[
1=A_0\leq A_1\leq \dots\leq  A_{\al}\leq A_{\al+1}\leq \dots A_{\ga}
\]
where $A_1 = \zeta_G(A) = \{a\in A| [a,g]=1\text{ for all }g\in G\}, A_{\al+1}/A_{\al}= \zeta_{G/A_{\al}}(A/A_{\al})$ for all ordinals $\al <\ga$, $A_{\la}=\cup_{\be<\la}A_{\be}$ for all limit ordinals $\la<\ga$ and $\ze_{G/A_{\ga}}(A/A_{\ga})=1$. We note that every subgroup in this series is $G$-invariant.When we wish to be more specific we denote the term $A_{\al}$ by $\ze_{G,\al}(A)$.

The last term $A_{\ga}$ of this series is called \emph{the upper $G$-hypercentre} of $A$ and is denoted by $\ze_{G,\infty}(A)$.

If  $A=A_{\gamma}$,  then  A  is called  \emph{$G$-hypercentral}; if  $\gamma$  is finite, then  $A$ is called  
\emph{$G$-nilpotent}.

A normal subgroup $A$ of $G$ is said to be  \emph{$G$-hypereccentric} if there is an ascending series 
 \[    
1 \leq C_0 \leq   C_1  \leq \dots  \leq C_{\alpha} \leq \dots\leq C_{\gamma} = A
\]
of  $G$-invariant subgroups   such that each factor  $C_{\alpha + 1}/C_{\alpha}$   is a  $G$-eccentric $G$-chief factor, for every  $\alpha <\gamma$. 

With these definitions concluded we may now state the following immediate corollary of Lemma~\ref{l:3}.

\begin{corollary}\label{c:4}
Let $G$ be a group and let $A$ be a normal subgroup of $G$. If $A$ is $G$-hypereccentric, then $[A,G]=A$.
\end{corollary}

A normal subgroup $A$ of $G$ is said to be  \emph{$G$-hyperfinite} if it has an ascending series 
 \[    
1 \leq C_0 \leq   C_1  \leq \dots  \leq C_{\alpha} \leq \dots\leq C_{\gamma} = A
\]
of  $G$-invariant subgroups   such that each factor  $C_{\alpha + 1}/C_{\alpha}$   is finite,  for every  $\alpha <\gamma$. 

We say that the  normal abelian subgroup $A$  has \emph{the $Z(G)$-decomposition} if
\[
A = \zeta_{G, \infty} (A) \oplus \eta_{G, \infty}(A),
\]
where  $ \eta_{G, \infty}(A)$  is the maximal  $G$-hypereccentric  $G$-invariant subgroup of $A$. This concept was introduced by D. I. Zaitsev~\cite{dZ79}.

The intersection $\mon_G(A)$ of all nontrivial $G$-invariant subgroups of $A$ is called the \emph{$G$-monolith} of $A$.  We say that $A$ is \emph{$G$-monolithic} if its $G$-monolith is nontrivial. In this case the $G$-monolith of $A$ is the unique minimal $G$-invariant subgroup of $A$.

\begin{lemma}\label{l:5}
Let $G$ be a group and let $A$ be an abelian normal subgroup of $G$ such that  $G/A$ is a finitely generated nilpotent group. Suppose that $A$ is $G$-monolithic with $G$-monolith $M$. If $M$ is not $G$-central, then $A$ is $G$-hyperfinite and $G$-hypereccentric.
\end{lemma}

\begin{proof}

Since $G/A$ is finitely generated there is a finitely generated subgroup $S$ such that $G=SA$ and since $A$ is abelian,  a subgroup of $A$ is $G$-invariant if and only it it is $S$-invariant.

Let $1\neq w\in M$.  %Since $M$ is a minimal normal subgroup of $G$ we have $M=\cyclic{w}^G=\cyclic{w}^S$. 
Let $X$ be an arbitrary finite subset of $A$ containing $w$ and let $B=\cyclic{X}^G=\cyclic{X}^S$.  Clearly $M\leq B$. Set $V=BS$ so that $V$ is a finitely generated abelian-by-nilpotent group. Let 
\[
\mathcal{S}=\{Y| Y\norm V \text{ and } Y\cap M=1\}
\]
and let $R$ be a maximal element of $\mathcal{S}$, which we order by inclusion. The maximal choice of $R$ shows that every nontrivial normal subgroup of $V/R$ has nontrivial intersection with $MR/R$ and it follows that $V/R$ is monolithic with monolith $MR/R$. By a result of P. Hall~\cite[Theorem 1*]{pH59} it follows that $V/R$ is finite.  Furthermore, $R\cap B$ is $G$-invariant and $R\cap M=1$, so $R\cap B=1$.  Thus $B\cong BR/R$ which implies that $B$ is finite. Hence for every finite subset $X$ of $A$ the subgroup $\cyclic{X}^G$ is finite.

Since $B$ is finite it has a $Z(G)$-decomposition, so $B=\ze_{G,\infty}(B)\oplus \eta_{G,\infty}(B)$ (see \cite[Corollary 1.6.5]{DKS17}). However $M$ is not $G$-central, so $\ze_{G,\infty}(A)=1$ and it follows that $B$ is $G$-hypereccentric. Thus $B$ has a finite series consisting of $G$-invariant subgroups whose factors are finite and $G$-eccentric. It is now easy to see that $A$ is $G$-hyperfinite and $G$-hypereccentric.
\end{proof}

We may immediately apply Lemma~\ref{l:5} and Corollary~\ref{c:4} to deduce the following.

\begin{corollary}\label{c:6}
Let $G$ be a group and let $A$ be an abelian normal subgroup of $G$ such that  $G/A$ is a finitely generated nilpotent group. Suppose that $A$ is $G$-monolithic with $G$-monolith $M$. If $M$ is not $G$-central, then $[A,G]=A$.
\end{corollary}

\begin{lemma}\label{l:7}
Let $G$ be a group and let $A$ be an abelian normal subgroup of $G$ such that  $G/A$ is a finitely generated nilpotent group.  If $G$ is not locally nilpotent, then $G$ contains a proper contranormal subgroup.
\end{lemma}

\begin{proof}
Since $G/A$ is finitely generated there is a finitely generated subgroup $S$ such that $G=SA$. Since $A$ is abelian, a subgroup of $A$ is $G$-invariant if and only if it is $S$-invariant. Factoring by the normal subgroup $S\cap A$ if necessary, we may assume that $S\cap A=1$.

Since $G$ is not locally nilpotent there is a finite subset $M$ of $A$ such that $B=\cyclic{M}^G=\cyclic{M}^S$ is not $S$-nilpotent. It follows from \cite[Corollary 13.2]{KOS02} that $B$ contains a maximal $G$-invariant subgroup $C$ with the property that $B/C$ is $G$-eccentric. The choice of $C$ implies that  $B/C$ is a minimal $G$-invariant subgroup of $A/C$.

Let
\[
\mathcal{S}=\{Y| Y\norm G, Y\leq A \text{ and }Y\cap B=C\}.
\]
Ordering $\mathcal{S}$ by inclusion, we may choose a maximal element $D$ of $\mathcal{S}$ and clearly $D\neq A$. If $E$ is  a $G$-invariant subgroup of  $A$ strictly containing $D$, then the minimality of $B/C$ implies that $E\cap B=B$. In particular, this implies that $A/D$ is $G$-monolithic with $G$-monolith $BD/D$.  Corollary~\ref{c:6} implies that $A/D=[G/D,A/D]=[G,A]D/D$ so $A=[G,A]D$. But $G=SA$ and $A$ is abelian so we have $A=[S,A]D$. 

Let $K=SD$. We claim that $K$ is a proper contranormal subgroup of $G$.

Indeed, if $K=G$,  then $A=A\cap SD=D(A\cap S)=D$, a contradiction. Also 
\[
K^G=S^GD=S^AD=S[S,A]D=SA=G,
\]
as required, using Lemma~\ref{l:2}.
\end{proof}

We state the clear consequence of this result.

\begin{corollary}\label{c:8}
Let $G$ be a group and let $A$ be an abelian normal subgroup of $G$ such that  $G/A$ is a finitely generated nilpotent group.  If $G$ is contranormal-free,  then $G$ is locally nilpotent.
\end{corollary}

\section{Proofs of the main results}

In order to prove Theorem~\ref{t:1} we require the following result which is interesting in its own right and strengthens Corollary~\ref{c:8}.
\begin{theorem}\label{t:1a}
Let $G$ be a group and let $L$ be a nilpotent normal subgroup of $G$ such that $G/L$ is a finitely generated nilpotent group.  If $G$ is contranormal-free, then $G$ is hypercentral. 
\end{theorem}

\begin{proof}
Let $K=L'$.  By Lemma~\ref{l:1}(iii) it follows that $G/K$ contains no proper contranormal subgroups either so Corollary~\ref{c:8} shows that $G/K$ is locally nilpotent.  Thus $G$ is locally nilpotent by \cite[Theorem 3]{bP61b}.

Let $1=A_0\leq A_1\leq \dots \leq A_t=L$ denote the upper central series of $L$.

Since $G/L$ is finitely generated, there is a finitely generated subgroup $S$ such that $G=LS$.  A subgroup of $A_1$ is $S$-invariant if and only if it is $G$-invariant, because $L\leq C_G(A_1)$.

Let $M$ be an arbitrary finite subset of $A_1$. The subgroup $\cyclic{M,S}$ is a finitely generated nilpotent group. Let $B=\cyclic{M}^G=\cyclic{M}^S\leq A_1$. Clearly $B$ is a $G$-invariant finitely generated abelian subgroup of $G$. It is easy to see that for this reason there a natural number $k$ such that $B\leq \ze_{G,k}(A_1)$.  Since this is true for every finite subset of $A_1$ it follows that $A_1\leq \ze_{G,\om}(L)$.

An easy induction argument using similar reasoning shows that $L=\ze_{G, t\om}(L)$.  Since $G/L$ is nilpotent it now follows immediately that $G$ is hypercentral.
\end{proof}

There are two corollaries to Lemma~\ref{l:7} that can be deduced. They are also consequences of \cite{KOS09b} and \cite{dR70} so we merely state the second of these. The first has the following proof independent of those two papers.

\begin{corollary} \label{c:1}
Let $G$ be a finitely generated soluble-by-finite group. If $G$ is contranormal-free, then $G$ is nilpotent.
\end{corollary}

\begin{proof}
Let $H$ be a soluble normal subgroup of $G$ such that $G/H$ is finite. By Lemma~\ref{l:1}(iii) it follows that $G/H$ has no proper contranormal subgroups either so \cite[Satz III.3.11]{bH67} shows that $G/H$ is nilpotent. Consequently $G$ is actually soluble.  

Let $\ga_{t}(G)$ denote the $t$th term of the lower central series of $G$ and let $G^{(k)}$ denote the $k$th term of the derived series of $G$. Suppose that for some natural number $k$ there is a natural number $t$ such that  $\ga_t(G)\leq G^{(k)}$ (this of course being true for $k=2$). If $\ga_t(G)\neq 1$, then $\ga_{t}(G)'\lneqq \ga_t(G)$.  Then $G/\ga_t(G)'$ is an extension of an abelian group by a finitely generated nilpotent group . By Lemma~\ref{l:1}(iii) $G/\ga_t(G)'$ has no proper contranormal subgroups so Corollary~\ref{c:8} implies that $G/\ga_t(G)'$ is locally nilpotent. Since it is also finitely generated it follows that $G/\ga_t(G)'$ is nilpotent, so there is a natural number $l$ such that $\ga_l(G)\leq \ga_t(G)'\leq G^{(k+1)}$. It follows that, for all $n$, $G/G^{(n)}$ is nilpotent and since $G$ is soluble we have that $G$ is nilpotent. 
\end{proof}

\begin{corollary} \label{c:2}
Let $G$ be a finitely generated group and suppose that $G$ has an ascending series of normal subgroups whose factors are abelian or finite. If $G$ is contranormal-free, then $G$ is nilpotent.
\end{corollary}

\begin{proof}[\textbf{Proof of Theorem~\ref{t:1}}]
The group $G/L$ is contranormal-free and finitely generated hyperabelian so by Corollary~\ref{c:2}, $G/L$ is nilpotent. Theorem~\ref{t:1a} now gives the result.
\end{proof}

\begin{proof}[\textbf{Proof of Theorem~\ref{t:2}}]
It follows from \cite[Theorem 4.2.1]{DKS17} that the group $G$ has a series of normal subgroups $T\leq S\leq G$ such that $T$ is locally finite with Chernikov Sylow $p$-subgroups for all primes $p$, $S/T$ is a torsion-free soluble group of finite $0$-rank and $G/S$ is finite.  Indeed, by Lemma~\ref{l:1} and \cite[Satz III.3.11]{bH67} $G/T$ is soluble.

Let $p\in \Pi(T)$.  Then $T/O_{p'}(T)$ is Chernikov (see \cite[Theorems 3.3.5 and 3.3.11]{DKS17}, for example). Certainly $O_{p'}(T)$ is $G$-invariant and Lemma~\ref{l:1}(iii) implies that $G/O_{p'}(T)$ contains no proper contranormal subgroups. It follows that $G/O_{p'}(T)$ is nilpotent by \cite[Theorem C]{KOS09} and this is true for all primes $p$. Since $\cap_{p\in\Pi(T)}O_{p'}(T)=1$, Remak's theorem shows that $G$ embeds in the Cartesian product $ \underset{p\in\Pi(T)}{\text{Cr}} G/O_{p'}(T)$. We note that a locally finite subgroup of a Cartesian product of nilpotent groups is locally nilpotent. Hence $T$ is locally nilpotent and $T=\dir{p}{\Pi(T)}{S}$, where $S_p$ is the Sylow $p$-subgroup of $T$. Clearly $O_{p'}(T)= \underset{q\neq p}{\text{Dr}}S_q$ and $T/O_{p'}(T)\cong S_p$. Since $G/O_{p'}(T)$ is nilpotent it follows that $S_p$ is nilpotent for each prime $p$. Let $n(p)$ be the nilpotency class of $G/O_{p'}(T)$.  Then we have $[S_p,_{n(p)} G]\leq O_{p'}(T)$. On the other hand $S_p$ is $G$-invariant so $[S_p, _{n(p)} G]\leq S_p$ and it follows that $[S_p,_{n(p)} G]=1$.This means that $S_p\leq \ze_{n(p)}(G)$ and since this is true for each prime $p$ it follows that $T\leq \ze_{\om}(G)$.  However $G/T $ is nilpotent so there is a natural number $k$ such that $G=\ze_{\om+k}(G)$ and $G$ is hypercentral of hypercentral length at most $\om+k$ as required.

Finally, we note that if $H$ is normal in $G$ and $\Pi(G/H)$ is finite, then this proof shows that $G/H$ is nilpotent.
\end{proof}

In conclusion, we note that Wehrfritz~\cite{bW20} has constructed a non-nilpotent (locally cyclic)-by-infinite cyclic group of special rank $2$ which is contranormal-free.  Furthermore, the well-known Heineken-Mohamed groups~\cite{HM68} are contranormal-free but of infinite section rank and are not hypercentral. Thus our theorems are best possible in that sense.

\providecommand{\bysame}{\leavevmode\hbox to3em{\hrulefill}\thinspace}
\providecommand{\MR}{\relax\ifhmode\unskip\space\fi MR }
\providecommand{\MRhref}[2]{%
  \href{http://www.ams.org/mathscinet-getitem?mr=#1}{#2}
}
\providecommand{\href}[2]{#2}


\begin{thebibliography}{10}

\bibitem{DKS17}
M.~R. Dixon, L.~A. Kurdachenko, and I.~Ya. Subbotin, \emph{Ranks of {G}roups:
  The {T}ools, {C}haracteristics, and {R}estrictions}, John Wiley \& Sons,
  Inc., Hoboken, NJ, 2017.

\bibitem{pH59}
P.~Hall, \emph{On the finiteness of certain soluble groups}, Proc. London Math.
  Soc. (3) \textbf{9} (1959), 595--622.

\bibitem{HM68}
H.~Heineken and I.~J. Mohamed, \emph{A group with trivial centre satisfying the
  normalizer condition}, J. Algebra \textbf{10} (1968), 368--376.

\bibitem{bH67}
B.~Huppert, \emph{Endliche {G}ruppen. {I}}, Die Grundlehren der Mathematischen
  Wissenschaften, Band 134, Springer-Verlag, Berlin, 1967.

\bibitem{KOS02}
L.~A. Kurdachenko, J.~Otal, and I.~Ya. Subbotin, \emph{Groups with {P}rescribed
  {Q}uotient {G}roups and {A}ssociated {M}odule {T}heory}, Series in Algebra,
  World Scientific, Singapore, 2002, Volume 8.

\bibitem{KOS09b}
L.~A. Kurdachenko, J.~Otal, and I~Ya. Subbotin, \emph{Criteria of nilpotency
  and influence of contranormal subgroups on the structure of infinite groups},
  Turkish J. Math. \textbf{33} (2009), no.~3, 227--237.

\bibitem{KOS09}
L.~A. Kurdachenko, J.~Otal, and I.~Ya. Subbotin, \emph{On influence of
  contranormal subgroups on the structure of infinite groups}, Comm. Algebra
  \textbf{37} (2009), no.~12, 4542--4557.

\bibitem{KS03b}
L.~A. Kurdachenko and I.~Ya. Subbotin, \emph{Pronormality, contranormality and
  generalized nilpotency in infinite groups}, Publ. Mat. \textbf{47} (2003),
  no.~2, 389--414.

\bibitem{bP61b}
B.~I. Plotkin, \emph{Some properties of automorphisms of nilpotent groups},
  Dokl. Akad. Nauk SSSR \textbf{137} (1961), 1303--1306.

\bibitem{dR70}
D.~J.~S. Robinson, \emph{A theorem on finitely generated hyperabelian groups},
  Invent. Math. \textbf{10} (1970), 38--43.

\bibitem{dR96}
\bysame, \emph{A {C}ourse in the {T}heory of {G}roups}, Graduate Texts in
  Mathematics, vol.~80, Springer Verlag, Berlin, Heidelberg, New York, 1996.

\bibitem{jR68}
J.~S. Rose, \emph{Nilpotent subgroups of finite soluble groups}, Math. Z.
  \textbf{106} (1968), 97--112.

\bibitem{bW20}
B.~A.~F. Wehrfritz, \emph{Groups with no proper contranormal subgroups}, Publ.
  Mat. \textbf{64} (2020), no.~1, 183--194.

\bibitem{dZ79}
D.~I. Zaitsev, \emph{Hypercyclic extensions of abelian groups}, Groups defined
  by properties of a system of subgroups ({R}ussian), Akad. Nauk Ukrain. SSR,
  Inst. Mat., Kiev, 1979, pp.~16--37, 152.

\end{thebibliography}
\end{document}